\documentclass[a4paper,11pt]{article}
\usepackage{amsfonts}
\usepackage{amsmath}
\usepackage{amssymb}
\usepackage{amsthm}
\numberwithin{equation}{section}         %¹«Ê½ÐòºÅËæ½ÚºÅ¶ø±ä
\usepackage{indentfirst}                %°ÑÿС½ÚµÚÒ»¶ÎÉèΪÊ×ÐÐËõ¸ñ
\usepackage{verbatim}                    %×¢Ê͵ôÒ»¶ÎÎÄ×Ö
\usepackage{mathrsfs}
\usepackage{graphicx}
\theoremstyle{plain} \newtheorem{defi}{Definition}[section] %¶¨Òå»·¾³
\theoremstyle{plain} \newtheorem{Lem}{Lemma}[section]       %ÒýÀí»·¾³
\theoremstyle{plain}\newtheorem{theorem}{Theorem}[section]  %¶¨Àí»·¾³
\theoremstyle{plain}    %ÍÆÂÛ»·¾³
\theoremstyle{remark} \newtheorem{Rem}{Remarks}[section]    %ÆÀ×¢»·¾³
\theoremstyle{plain}  %ÃüÌâ»·¾³
\theoremstyle{plain}\newtheorem{example}{Example}[section]  %Àý×Ó»·¾³

\newcommand{\be}{\begin{equation}}
\newcommand{\ee}{\end{equation}}
\newcommand{\bpm}{\begin{pmatrix}}
\newcommand{\epm}{\end{pmatrix}}

\begin{document}

%\begin{title}
%\Huge Numerical Solutions of Jump Diffusion with Markovian Switching
%\end{title}
%\title{Numerical Solutions of Jump Diffusion with Markovian
%switching}
%\author{J. Ye}

\begin{center}
\LARGE{Strong Predictor-Corrector Euler-Maruyama Methods for Stochastic Differential Equations with Markovian Switching}\\
\end{center}
\begin{center}
Jun Ye, Haibo Li, Lili Xiao\\
\end{center}
\begin{center}
\emph{Department of Mathematical Sciences, Tsinghua University,
Beijing,
P.R.China}\\[.5cm]
\end{center}

\def\abstractname{Abstract}
\begin{abstract}
In this paper numerical methods for solving stochastic differential
equations with Markovian switching (SDEwMSs) are developed by
pathwise approximation. The proposed family of strong
predictor-corrector Euler-Maruyama methods is designed to overcome
the propagation of errors during the simulation of an approximate
path. This paper not only shows the strong convergence of the
numerical solution to the exact solution but also reveals the order
of the error under some conditions on the coefficient functions. A
natural analogue of $p$-stability criterion is studied. Numerical
examples are given to illustrate the computational
efficiency of the new predictor-corrector Euler-Maruyama approximation.\\

\end{abstract}

\section{Introduction}
Stochastic differential equations with Markovian switching (SDEwMSs)
arise in mathematics models of hybrid systems that possess frequent
unpredictable structural changes. One of the distinct features of
such systems is that the underlying dynamics are subject to changes
with respect to certain configurations. Such models have been used
with great success in a variety of application areas, including
flexible manufacturing systems, electric power networks, risk
theory, financial engineering and insurance modeling, we refer the
readers to Arapostathis, Ghosh and Marcus $\cite{AGM93}$, Jobert and
Rogers $\cite{JR6}$, Mao and Yuan $\cite{MY6}$, Rolski, Schmidli,
Schmidt and Teugels $\cite{RSST9}$, Smith $\cite{S2}$, Yang and Yin
$\cite{YY4}$ and references therein.

Generally, although the fundamental theories such as existence and
uniqueness of the solution as well as stability of SDEwMSs have been
well studied, most of SDEwMSs cannot be solved analytically. Thus,
appropriate numerical approximation methods such as the Euler (or
Euler-Maruyama) method are needed to apply SDEwMSs in practice or to
study their properties.

 Yuan and Mao $\cite{YM4}$ firstly considered
the numerical solutions of the following stochastic differential
equations with Markovian switching
\begin{equation}\label{eq:sdewms}
\begin{split}
&dy(t)=f(y(t),r(t))dt+g(y(t),r(t))dW(t),\\
\end{split}
\end{equation}
here $y(t)$ is referred to the state while $r(t)$ is regarded as the
mode. The system will switch from one mode to another in a random
way, and the switching between the modes is governed by a Markov
chain. They proved the mean-square convergence of the
Euler-Maruyama(EM) approximation for this hybrid stochastic systems,
and the order of error was also estimated. Yin, Song and Zhang
$\cite{YSZ5}$ extended (\ref{eq:sdewms}) to a family of more general
jump-diffusions with Markovian switching, and proved the numerical
solutions based on finite-difference procedure weak converge to the
desired limit by means of a martingale problem formulation.

During recent years, there also exist extensive literatures which
prove the convergence of the Euler-Maruyama method applied to some
stochastic differential equation with some additional feature, like
including some sort of delay, jumps, Markovian switching or
combinations thereof, see for example, Bruti-Liberati and Platen
$\cite{BP7}$, Hou, Tong and Zhang $\cite{HTZ9}$, Li and Hou
$\cite{LH6}$, Mao and Yuan $\cite{MY6}$, Rathinasamy and
Balachandran $\cite{RB8}$, among others. The corresponding proof is
basically the same each time, the only novelty coming from changing
it a bit to deal with the additional feature.

It is well known that Euler-Maruyama method and most other explicit
schemes for solving stochastic differential equations (SDEs) work
unreliably and sometimes generate large errors, see for instance
Milstein, Platen and Schurz $\cite{MPS8}$, implicit and
predictor-corrector schemes are designed to achieve improved
numerical stability and turn out to be better suited to simulation
task. Generally, implicit schemes usually cost significant
computational time and are sometimes not reliably accomplished, however,
this phenomenon can be avoided when using some appropriate discrete time
schemes, including predictor-corrector methods. In Kloeden and
Platen $\cite{KP99}$, predictor-corrector methods have been proposed
as weak discrete time approximations for solving SDEs, which
can be used in Monte Carlo simulation. For the strong discrete time
approximation of solutions of SDEs, a family of predictor-corrector
Euler methods has been developed in Bruti-Liberati and Platen
$\cite{BP8}$. However, there are no strong predictor-corrector
methods available for SDEwMSs yet.

In this paper, we develop a new family of strong predictor-corrector
Euler-Maruyama (PCEM) methods for SDEwMSs (\ref{eq:sdewms}), which
are shown to converge with strong order 0.5, and demonstrate their
performance by considering some examples.

The rest of the paper is arranged as follows. In Section 2 we
introduce some necessary notations and define a family of strong
predictor-corrector Euler-Maruyama approximate solutions to SDEwMSs.
In Section 3 we show that the PCEM solutions converge to the exact
solution in $L^2$ under the global Lipschitz condition and reveal
the order of convergence is 0.5. In Section 4 we extend the PCEM
convergence results to multi-dimensional case under certain
conditions. In Section 5 the numerical stability of SDEwMSs will be
introduced and discussed. Finally, in Section 6 some numerical
examples are given and compared for simulated paths with different
degrees of implicitness to illustrate the computational efficiency
of the predictor-corrector Euler-Maruyama approximation.

\section{Preliminary and algorithm}
Let $(\Omega,\mathcal{F},\{\mathcal{F}_t\}_{t\geq0},P)$ be a
complete probability space with a filtration
$\{\mathcal{F}_t\}_{t\geq0}$ satisfying the usual conditions.
Suppose that there is a finite set $S=\{1,2,...,N\}$, representing
the possible regimes of the environment. We work with a finite-time
horizon $[0,T]$ for some $T>0$. Let $f(\cdot,\cdot):\mathbb{R}^d
\times S \rightarrow \mathbb{R}^d$, $g(\cdot,\cdot):\mathbb{R}^d
\times S \rightarrow \mathbb{R}^{d\times d}$ be both Borel
measurable. Consider the dynamic system given by (1.1) with initial
value $y(0)=y_0\in\mathbb{R}^d$ and $r(0)=i_0\in S$, where
$W(t)=(W^{1}(t),\cdots,W^{d}(t))^T$ is a $d$-dimensional
$\mathcal{F}_t$-adapted standard Brownian motion, and $r(t)$ is a
continuous-time Markov chain taking value in a finite state space
$S=\{1,2,...,N\}$ with the generator $Q=(q_{ij})_{N\times N}$ given
by
\begin{equation*}
P\{r(t+\delta)=j|r(t)=i\}=\begin{cases}
q_{ij}\delta+o(\delta), &\text{if $i\neq j$},\\[.5cm]
1+q_{ii}\delta+o(\delta), &\text{if $i=j$},
\end{cases}
\end{equation*}
provided $\delta\downarrow 0$, and
\[-q_{ii}=\sum_{i\neq j}q_{ij}<+\infty.\]

 We assume $W(t)$ and $r(t)$ are independent. Throughout this paper,
 we denote by $|\cdot|$ the Euclidean norm for vectors or the trace
 norm for matrices.

\subsection{Existence and uniqueness}
Under certain conditions we can establish the existence of a
pathwise unique solution of (\ref{eq:sdewms}). Here we make the following global
Lipschitz (GL) and linear growth (LG) assumptions:\\[.2cm]
($\mathcal{H}$1) GL: For all $\ (x,i),(y,i)\in \mathbb{R}^d\times
S$, there exists a constant $L_1>0$ such that
\[|f(x,i)-f(y,i)|^2+|g(x,i)-g(y,i)|^2
\leq L_1|x-y|^2.\]\\
($\mathcal{H}$2) LG: For all $\ (x,i)\in
\mathbb{R}^d\times S$, there exists a constant $L_2>0$ such that
\[|f(x,i)|^2\vee |g(x,i)|^2\leq L_2(1+|x|^2).\]

\begin{Rem}It is easy to show that if $f(\cdot,\cdot),\
g(\cdot,\cdot)$ satisfy GL condition, then they also satisfy LG
condition, but for the convenience of the reader we preserve it.
\end{Rem}

The following theorem guarantee the existence and uniqueness of the
solution to equation (\ref{eq:sdewms}), which are of use in studying
some numerical schemes.
\begin{theorem}
If $f(x,i)$, $g(x,i)$ satisfy the conditions ($\mathcal{H}$1),
($\mathcal{H}$2), and suppose $W(t),r(t)$ be independent. Then there
exists a unique d-dimensional $\mathcal{F}_t$-adapted
right-continuous process $y(t)$ with left-hand limits which
satisfies equation (\ref{eq:sdewms}) such that $y(0)=y_0$ and $r(0)=i_0$ a.s.
\end{theorem}

\begin{proof}
See Theorem 3.13 in Mao and Yuan $\cite{MY6}$.
\end{proof}

\subsection{Algorithm} Now we turn our attention to numerical
algorithm. For convenience, we first consider one-dimensional
SDEwMSs. Given $\Delta>0$ as a step size, denote
$\{t_i\}_{i\geq1}$ the usual equidistant time discretization of a
bounded interval $[0,T]$, i.e. $t_0=0$,
 $t_i-t_{i-1}=\Delta$, if $t_{n-1}<T\leq t_n$ then set
 $t_n=T$. Denote $\Delta W_{t_k}=W({t_{k+1}})-W({t_k})$.

For given partition $\{t_k\}_{k\geq 1}$, $\{r(t_k)\}_{k\geq1}$ is a
discrete Markov chain with transition probability matrix
$(P(i,j))_{N\times N}$, here $P(i,j)= P(r(t_{k+1})=j|r(t_{k})=i)$ is
the $ij$th entry of the matrix $e^{(t_{k+1}-t_{k})Q}$, thus we could
use following recursion procedure to simulate the discrete Markov
chain $\{r(t_k)\}_{k\geq1}$, suppose $r(t_k)=i_1$ and generate a
random number $\xi$ which is uniformly distributed in $[0,1]$, then
we define
\begin{equation*} r(t_{k+1})=\begin{cases}
i_2, &\text{if $\ i_2\in S-\{N\}$ and $\ \sum^{i_2-1}_{j=1}P(i_1,j)\leq\xi<\sum^{i_2}_{j=1}P(i_1,j)$},\\[.5cm]
N, &\text{if $\ \sum^{N-1}_{j=1}P(i_1,j)\leq\xi.$}
\end{cases}
\end{equation*}
Repeating this procedure a trajectory of $\{r(t_k)\}_{k\geq1}$ can
be simulated.

Now we can introduce a new family of strong predictor-corrector
Euler-Maruyama(PCEM) methods for SDEwMSs. Given initial value
$Y_{t_{0}}=y_0\in\mathbb{R}$ and $r_{t_{0}}=i_0\in S$, the proposed
family of strong PCEM is given by the predictor
\begin{equation}\label{eq:pcemp}
\begin{split}
\widetilde{Y}_{t_{k+1}}&=Y_{t_k}+f(Y_{t_k},r_{t_k})\Delta+g(Y_{t_k},r_{t_k})\Delta W_{t_k},\\
\end{split}
\end{equation}
and by the corrector
\begin{equation}\label{eq:pcemc}
\begin{split}
Y_{t_{k+1}}=&Y_{t_k}+\{\theta
\bar{f}_\eta(\widetilde{Y}_{t_{k+1}},r_{t_k})+(1-\theta)\bar{f}_\eta(Y_{t_k},r_{t_k})\}\Delta\\
&+\{\eta g(\widetilde{Y}_{t_{k+1}},r_{t_k})+(1-\eta)g(Y_{t_k},r_{t_k})\}\Delta W_{t_k}.\\
\end{split}
\end{equation}
Here parameters $\theta, \eta\in[0,1]$ denote the degree of
implicitness in the drift and the diffusion coefficients,
respectively, and $\bar{f}_\eta(x,i)$ is defined as
\begin{equation}\label{eq:cdfun}
\bar{f}_\eta(x,i)=f(x,i)-\eta g(x,i)\frac{\partial g(x,i)}{\partial
x},\eta\in[0,1],
\end{equation}
which is called the corrected drift function. This scheme can be
written in the form
\begin{equation}\label{eq:pcems}
\begin{split}
Y_{t_{k+1}}&=Y_{t_k}+f(Y_{t_k},r_{t_k})\Delta+g(Y_{t_k},r_{t_k})\Delta W_{t_k}+\sum_{l=1}^{4}R_{l,k},\\
\end{split}
\end{equation}
where
\begin{equation}\label{eq:pcemr1}
\begin{split}
R_{1,k}&=\theta \{f(\widetilde{Y}_{t_{k+1}},r_{t_k})-f(Y_{t_k},r_{t_k})\}\Delta,\\
\end{split}
\end{equation}
\begin{equation}\label{eq:pcemr2}
\begin{split}
R_{2,k}&=-\theta\eta
\{g(\widetilde{Y}_{t_{k+1}},r_{t_k})\frac{\partial
g(\widetilde{Y}_{t_{k+1}},r_{t_k})}{\partial
x}-g(Y_{t_k},r_{t_k})\frac{\partial g(Y_{t_k},r_{t_k})}{\partial
x}\}\Delta,\\
\end{split}
\end{equation}
\begin{equation}\label{eq:pcemr3}
\begin{split}
R_{3,k}&=-\eta g(Y_{t_k},r_{t_k})\frac{\partial
g(Y_{t_k},r_{t_k})}{\partial
x}\Delta,\\
\end{split}
\end{equation}
\begin{equation}\label{eq:pcemr4}
\begin{split}
R_{4,k}&=\eta
\{g(\widetilde{Y}_{t_{k+1}},r_{t_k})-g(Y_{t_k},r_{t_k})\}\Delta W_{t_k}.\\
\end{split}
\end{equation}

For each $t\in [t_k,t_{k+1})$, let
\begin{equation}\label{eq:instead}
\begin{split}
\bar{Y}(t)=Y_{t_k},\bar{r}(t)=r_{t_k},\bar{R}_l
(t)=\frac{R_{l,k}}{\Delta},l=1,2,3, \bar{R}_4
(t)=\frac{R_{4,k}}{\Delta W_{t_k}}.\
\end{split}
\end{equation}

Therefore, we can define the continuous approximation solution
$Y(t)$ on the entire interval $[0,T]$ by
\begin{equation}\label{eq:cas}
\begin{split}
Y(t)=y_0&+\int_{0}^{t}f(\bar{Y}(s),\bar{r}(s))ds+\int_{0}^{t}g(\bar{Y}(s),\bar{r}(s))dW(s)\\
&+\sum_{l=1}^{3}\int_{0}^{t}\bar{R}_l(s)ds+\int_{0}^{t}\bar{R}_4(s)dW(s).
\end{split}
\end{equation}

Note that $Y(t_k)=\bar{Y}(t_k)=Y_{t_k}$, which means $Y(t)$ and
$\bar{Y}(t)$ coincide with the discrete approximate solution at the
gridpoints.

\begin{Rem}
The major advantage of the above PCEM approximate schemes is that
there are flexible degrees of implicitness parameters $\theta$ and
$\eta$ to choose for simulating paths properly. For the case
$\theta=\eta=0$ one recovers the Euler-Maruyama scheme which is well
discussed in Yuan and Mao $\cite{YM4}$.
\end{Rem}

\section{Convergence with the global Lipschitz(GL) and linear growth(LG) conditions}
In this section, we will prove that the numerical solution $Y(t)$
converges to the exact solution $y(t)$ in $L^2$ as step size
$\Delta\downarrow0$, and the order of convergence is one-half. To
begin with, we need the following GL condition and LG condition for
$\bar{f}_\eta(\cdot,\cdot)$.\\

\noindent($\mathcal{H}$3) GL: For all $\ (x,i),(y,i)\in
\mathbb{R}\times S$, there exists a constant $L_3>0$ such that
\[|\bar{f}_\eta(x,i)-\bar{f}_\eta(y,i)|^2\leq L_3|x-y|^2.\]
\noindent($\mathcal{H}$4) LG: For all $\ (x,i)\in \mathbb{R}\times
S$, there exists a constant $L_4>0$ such that
\[|\bar{f}_\eta(x,i)|^2\leq L_4(1+|x|^2).\]

We are now ready to present the key results of this section which
are stated as following.
\begin{theorem}\label{theo31}
Assume the SDEwMSs (\ref{eq:sdewms}) defined on $(\Omega,\mathcal{F},\{\mathcal{F}_t\}_{t\geq0},P)$ satisfying\\[.5cm]
\textbf{a}) $W(t),r(t)$ are independent,\\[.5cm]
\textbf{b}) $f(\cdot,\cdot), g(\cdot,\cdot), \bar{f}_\eta(\cdot,\cdot)$ satisfy conditions
$(\mathcal{H}1)$, $(\mathcal{H}2)$, $(\mathcal{H}3)$ and $(\mathcal{H}4)$, \\[.5cm]
then the unique strong solution $y(t)$ and numerical solution $Y(t)$
obtained in section 2.2 satisfying:

\begin{equation}\label{eq:theo1}
E(\sup_{0\leq t\leq T}|Y(t)-y(t)|^2)\leq C\Delta+o(\Delta),
\end{equation}
where $C$ is a positive constant independent of $\Delta$.
\end{theorem}

In order to give the proof of this theorem, we first provide a
number of useful lemmas. The first two lemmas show that the
continuous approximation has bounded moments in a strong sense. The
latter lemmas play an important role in proving the strong
convergence result, which mainly refer to Bruti-Liberati and Platen
$\cite{BP8}$.

\begin{Lem}\label{lem1}
Under conditions $(\mathcal{H}1)$, $(\mathcal{H}2)$,
$(\mathcal{H}3)$ and $(\mathcal{H}4)$, for any $p\geq 2$, there
exists a constant $M$ which is dependent on $p$, $T$, $L$ and $y_0$,
but independent of $\Delta$, such that
\begin{equation}\label{eq:lem1}
E(\sup_{0\leq t\leq T}|Y(t)|^p)\leq M.
\end{equation}
\end{Lem}
\noindent We omit the proof since it is similar to Lemma 4.1 in Mao
and Yuan $\cite{MY6}$.

\begin{Lem}\label{lem2}
Under conditions $(\mathcal{H}1)$, $(\mathcal{H}2)$,
$(\mathcal{H}3)$ and $(\mathcal{H}4)$, there exists a constant $M$
which is dependent on $T$, $L$ and $y_0$, but independent of
$\Delta$, such that
\begin{equation}\label{eq:lem2}
E(\max_{0\leq k\leq n_T}|Y_{t_k}|^2)\leq M.
\end{equation}
\end{Lem}
\noindent This is an immediate result of Lemma \ref{lem1}, since
$Y(t_k)=\bar{Y}(t_k)=Y_{t_k}$.

\begin{Lem}\label{lem3}
There exists a constant $C$ which is dependent on $T$, $L$ and
$y_0$, but independent of $\Delta$, such that
\begin{equation}\label{eq:lem3}
\sum_{l=1}^{4}E[\max_{1\leq n\leq n_T}|\sum_{0\leq k\leq
n-1}R_{l,k}|^2]\leq C\Delta.
\end{equation}
\end{Lem}

\begin{proof}
By the Cauchy-Schwarz inequality and the GL condition, from equation
(\ref{eq:pcemr1}), we can obtain
\begin{equation}\label{eq:lem31}
\begin{split}
&E[\max_{1\leq n\leq n_T}|\sum_{0\leq k\leq
n-1}R_{1,k}|^2]\\
&=E[\max_{1\leq n\leq n_T}|\sum_{0\leq k\leq
n-1}\theta \{f(\widetilde{Y}_{t_{k+1}},r_{t_k})-f(Y_{t_k},r_{t_k})\}\Delta|^2]\\
&\leq E[\max_{1\leq n\leq n_T}(\sum_{0\leq k\leq
n-1}|\theta\Delta|^2)(\sum_{0\leq k\leq
n-1}|f(\widetilde{Y}_{t_{k+1}},r_{t_k})-f(Y_{t_k},r_{t_k})|^2)]\\
&\leq C\Delta E[(\sum_{0\leq k\leq n_T-1}\Delta)(\sum_{0\leq k\leq
n_T-1}|f(\widetilde{Y}_{t_{k+1}},r_{t_k})-f(Y_{t_k},r_{t_k})|^2)]\\
&\leq C\Delta E[\sum_{0\leq k\leq
n_T-1}|\widetilde{Y}_{t_{k+1}}-Y_{t_k}|^2]\\
&\leq C\Delta E[\sum_{0\leq k\leq n_{T-1}}
(E[|f(Y_{t_k},r_{t_k})\Delta|^2|\mathcal{F}_{t_k}]+E[|g(Y_{t_k},r_{t_k})\Delta
W_{t_k}|^2|\mathcal{F}_{t_k}])].\\
\end{split}
\end{equation}
Then by using the Cauchy-Schwarz inequality, the It\^{o}'s isometry,
the LG condition and Lemma \ref{lem2}, we get
\begin{equation}\label{eq:lem32}
\begin{split}
&E[\max_{1\leq n\leq n_T}|\sum_{0\leq k\leq
n-1}R_{1,k}|^2]\\
&\leq C\Delta
E[\int_0^{t_{n_T}}E((1+|Y_{t_{n_z}}|^2)|\mathcal{F}_{t_{n_z}})dz]\\
&\leq C\Delta
\int_0^{t_{n_T}}E(1+\max_{0\leq n\leq n_T}|Y_{t_n}|^2)dz\\
&\leq C\Delta.\\
\end{split}
\end{equation}
With the similar steps as in (\ref{eq:lem31}) and (\ref{eq:lem32}),
we have
\begin{equation}\label{eq:lem33}
\begin{split}
&E[\max_{1\leq n\leq n_T}|\sum_{0\leq k\leq n-1}R_{2,k}|^2]\leq \
C\Delta.\\
\end{split}
\end{equation}
It is also easy to show by the LG condition and Lemma \ref{lem2}
that
\begin{equation}\label{eq:lem34}
\begin{split}
&E[\max_{1\leq n\leq n_T}|\sum_{0\leq k\leq n-1}R_{3,k}|^2]\leq \
C\Delta.\\
\end{split}
\end{equation}
By using Doob's martingale inequality, the It\^{o}'s isometry, the
GL condition and equation (\ref{eq:pcemr4}), we have
\begin{equation}\label{eq:lem35}
\begin{split}
&E[\max_{1\leq n\leq n_T}|\sum_{0\leq k\leq n-1}R_{4,k}|^2]\\
&=E[\max_{1\leq n\leq n_T}|\sum_{0\leq k\leq
n-1}\eta\int_{t_k}^{t_k+1} (g(\widetilde{Y}_{t_{k+1}},r_{t_k})-g(Y_{t_k},r_{t_k}))dW(z)|^2]\\
&\leq CE[|\int_0^{T}(g(\widetilde{Y}_{t_{n_z+1}},r_{t_{n_z}})-g(Y_{t_{n_z}},r_{t_{n_z}}))dW(z)|^2]\\
&= C\int_0^{T}E[|g(\widetilde{Y}_{t_{n_z+1}},r_{t_{n_z}})-g(Y_{t_{n_z}},r_{t_{n_z}})|^2]dz\\
&\leq C\int_0^{T}E[|\widetilde{Y}_{t_{n_z+1}}-Y_{t_{n_z}}|^2]dz.\\
\end{split}
\end{equation}
Therefore, with similar steps as in (\ref{eq:lem31}) and
(\ref{eq:lem32}), we also have
\begin{equation}\label{eq:lem36}
\begin{split}
&E[\max_{1\leq n\leq n_T}|\sum_{0\leq k\leq n-1}R_{4,k}|^2]\\
&\leq C\int_0^{T}E[\int_{t_{n_z}}^{t_{n_z+1}}E[(1+|Y_{t_{n_z}}|^2)|\mathcal{F}_{t_{n_z}}]ds]dz\\
&\leq C\int_0^{T}E[\int_{t_{n_z}}^{t_{n_z+1}}E[(1+\max_{0\leq n\leq n_T}|Y_{t_n}|^2)|\mathcal{F}_{t_{n_z}}]ds]dz\\
&\leq C\Delta.\\
\end{split}
\end{equation}
Thus the required assertion follows.
\end{proof}

\begin{Lem}\label{lem4}
Under conditions $(\mathcal{H}1)$, $(\mathcal{H}2)$,
$(\mathcal{H}3)$ and $(\mathcal{H}4)$, then for any
$t\in[t_k,t_{k+1})$, we have
\begin{equation}\label{eq:lem4}
\sum_{l=1}^{3}E[\sup_{t_k\leq s\leq t}|\int_{t_k}^s
\frac{R_{l,k}}{\Delta}dz|^2]+E[\sup_{t_k\leq s\leq t}|\int_{t_k}^s
\frac{R_{4,k}}{\Delta W_{t_k}}dW(z)|^2]\leq o(\Delta).\\
\end{equation}
\end{Lem}
\noindent The proof of this lemma is similar to that in Lemma \ref{lem3}.\\

Now we are in a position to prove our Theorem \ref{theo31}.

\noindent $Proof\  of\  Theorem\  \ref{theo31}:$ From equation
(\ref{eq:cas}), we have
\begin{equation}\label{eq:theo11}
\begin{split}
Z(t)&=E(\sup_{0\leq s\leq t}|Y(s)-y(s)|^2)\\
&=E[\sup_{0\leq s\leq
t}|\int_{0}^{s}(f(\bar{Y}(z),\bar{r}(z))-f(y(z),r(z)))dz\\
&\quad+\int_{0}^{s}(g(\bar{Y}(z),\bar{r}(z))-g(y(z),r(z)))dW(z)\\
&\quad+\sum_{l=1}^{3}\int_{0}^{s}\bar{R}_l(z)dz+\int_{0}^{s}\bar{R}_4(z)dW(z)|^2].\\
\end{split}
\end{equation}
Let $n=[t/\Delta]$, the integer part of $t/\Delta$. Then, by
H\"{o}lder inequality, Doob's martingale inequality and equation
(\ref{eq:instead}), we have
\begin{equation}\label{eq:theo12}
\begin{split}
Z(t)=&E(\sup_{0\leq s\leq t}|Y(s)-y(s)|^2)\\
\leq &CE\int_{0}^{t}|f(\bar{Y}(z),\bar{r}(z))-f(y(z),r(z))|^2dz\\
&+CE\int_{0}^{t}|g(\bar{Y}(z),\bar{r}(z))-g(y(z),r(z))|^2dz\\
&+C\sum_{l=1}^{4}E[\max_{1\leq m\leq n}|\sum_{0\leq k\leq
m-1}R_{l,k}|^2]+C\sum_{l=1}^{3}E[\sup_{t_k\leq s\leq
t}|\int_{t_k}^s\frac{R_{l,k}}{\Delta}dz|^2]\\
&+CE[\sup_{t_k\leq s\leq t}|\int_{t_k}^s\frac{R_{4,k}}{\Delta
W_{t_k}}dW(z)|^2].\\
\end{split}
\end{equation}
We focus on the last three parts of the right side. From Lemma
\ref{lem3} and Lemma \ref{lem4}, we have
\begin{equation}\label{eq:theo13}
\begin{split}
&C\sum_{l=1}^{4}E[\max_{1\leq m\leq n}|\sum_{0\leq k\leq
m-1}R_{l,k}|^2]+C\sum_{l=1}^{3}E[\sup_{t_k\leq s\leq
t}|\int_{t_k}^s\frac{R_{l,k}}{\Delta}dz|^2]\\
&+CE[\sup_{t_k\leq s\leq t}|\int_{t_k}^s\frac{R_{4,k}}{\Delta
W_{t_k}}dW(z)|^2]\leq C\Delta +o(\Delta).\\
\end{split}
\end{equation}
Let $I_G$ be the indicator function for set $G$. Let
$t\in[t_k,t_{k+1})$. From (\ref{eq:pcemp})--(\ref{eq:cas}), obviously,
we have $\bar{Y}(t_{k})$ and $I_{\{r(t)\neq r(t_k)\}}$ are
conditionally independent with respect to the $\sigma$-algebra
generated by $r(t_k)$. So by the same procedures as in Theorem 3.1
in Yuan and Mao $\cite{YM4}$, we can obtain
\begin{equation}\label{eq:theo14}
\begin{split}
&E\int_{0}^{t}|f(\bar{Y}(s),\bar{r}(s))-f(y(s),r(s))|^2ds\\
&\quad\leq2L^2\int_{0}^{t}E|\bar{Y}(s)-y(s)|^2ds+C\Delta+o(\Delta),\\
\end{split}
\end{equation}
\begin{equation}\label{eq:theo15}
\begin{split}
&E\int_{0}^{t}|b(\bar{Y}(s),\bar{r}(s))-b(y(s),r(s))|^2ds \\
&\quad\leq2L^2\int_{0}^{t}E|\bar{Y}(s)-y(s)|^2ds+C\Delta+o(\Delta).\\
\end{split}
\end{equation}
Substituting (\ref{eq:theo13}), (\ref{eq:theo14}), (\ref{eq:theo15}) into (\ref{eq:theo12}) shows that
\begin{equation}\label{eq:theo16}
\begin{split}
E(\sup_{0\leq s\leq t}|Y(s)-y(s)|^2)&\leq
C\int_{0}^{t}E|\bar{Y}(s)-y(s)|^2ds+C\Delta+o(\Delta).
\end{split}
\end{equation}
Note that
\begin{equation}\label{eq:theo17}
E|\bar{Y}(s)-y(s)|^2\leq 2E|Y(s)-y(s)|^2+2E|\bar{Y}(s)-Y(s)|^2.
\end{equation}
Suppose $t_k\leq s< t_{k+1}$, by (\ref{eq:cas}), Lemma \ref{lem3} and Lemma \ref{lem4},
we can then show in the same way as in the case of SDEs that
\begin{equation}\label{eq:theo18}
E|Y(s)-\bar{Y}(s)|^2\leq C\Delta.\\
\end{equation}
Substituting (\ref{eq:theo17}), (\ref{eq:theo18}), into (\ref{eq:theo16}) immediately shows that
\begin{equation}\label{eq:theo19}
\begin{split}
&E(\sup_{0\leq s\leq t}|Y(s)-y(s)|^2)\\
&\quad\leq C\int_{0}^{t}(E|Y(s)-y(s)|^2+E|\bar{Y}(s)-Y(s)|^2)ds+C\Delta+o(\Delta)\\
&\quad\leq C\int_{0}^{t}E(\sup_{0\leq s\leq
t}|Y(s)-y(s)|^2)ds+C\Delta+o(\Delta).\\
\end{split}
\end{equation}
Therefore, from Gronwall inequality we obtain that
\[E(\sup_{0\leq t\leq T}|Y(t)-y(t)|^2)\leq C\Delta+o(\Delta).\]
The proof is complete.

\section{The general multi-dimensional case}
The results derived in Section 3 can be easily generalized to the
multi-dimensional case, we just summarize the related numerical
schemes and the convergence results in this section.

 Consider the solution
$y(t)=\{(y^1(t),...,y^d(t))^T, t\geq 0\}$ of the $d$-dimensional
SDEwMSs
\begin{equation*}
\begin{split}
y(t)=y(0)+\int_0^t f(y(s),r(s))ds+\sum_{j=1}^{m}\int_0^t g^j
(y(s),r(s))dW^j(s),\\
\end{split}
\end{equation*}
for $t\geq0$. Here $y(0)\in\mathbb{R}^d$ denotes the deterministic
initial value, $W^j=\{W^j(t), t\geq 0\}$, $j\in\{1,2,...,m\}$ is a
standard Brownian motion. $r(t)$ is a Markov chain. The function
$f:\mathbb{R}^d\times S\mapsto \mathbb{R}^d$ has the $k$th component
$f^k(\cdot,\cdot)$. The function $g^j:\mathbb{R}^d\times S\mapsto
\mathbb{R}^d, j\in{1,2,...,m}$ has the $k$th component
$g^{k,j}(\cdot,\cdot)$. We define the function $\bar{f}_{\eta}$ for
$\eta\in[0,1]$ with the $k$th component
\begin{equation*}
\begin{split}
\bar{f}_{\eta}^k(x,i)=f^k(x,i)-\eta_k \sum_{j_1,j_2=1}^{m}
\sum_{i=1}^{d}g^{i,j_1}(x,i)\frac{\partial g^{i,j_2}(x,i)}{\partial
x^i}.\\
\end{split}
\end{equation*}
for $(x,i)\in\mathbb{R}^d\times S$, which satisfies the following GL
condition and LG condition

\noindent($\mathcal{H}$3$'$) GL: For all $\ (x,i),(y,i)\in
\mathbb{R}^d\times S$, there exists a constant $L_3>0$ such that
\[|\bar{f}_\eta(x,i)-\bar{f}_\eta(y,i)|^2\leq L_3|x-y|^2.\]
\noindent($\mathcal{H}$4$'$)  LG: For all $\ (x,i)\in
\mathbb{R}^d\times S$, there exists a constant $L_4>0$ such that
\[|\bar{f}_\eta(x,i)|^2\leq L_4(1+|x|^2).\]

The $k$th component of the proposed family of strong PCEM schemes is
given by the predictor
\begin{equation*}
\begin{split}
\widetilde{Y}_{t_{n+1}}^k&=Y_{t_n}^k+f^k (Y_{t_n},r_{t_n})\Delta+\sum_{j=1}^m g^{k,j}(Y_{t_n},r_{t_n})\Delta W_{t_n}^j,\\
\end{split}
\end{equation*}
and by the corrector
\begin{equation*}
\begin{split}
Y_{t_{n+1}}^k=&Y_{t_n}^k+\{\theta_k
\bar{f}_\eta^k(\widetilde{Y}_{t_{n+1}},r_{t_n})+(1-\theta_k)\bar{f}_\eta^k(Y_{t_n},r_{t_n})\}\Delta\\
&+\sum_{j=1}^m\{\eta_k g^{k,j}(\widetilde{Y}_{t_{n+1}},r_{t_n})+(1-\eta_k)g^{k,j}(Y_{t_n},r_{t_n})\}\Delta W_{t_n},\\
\end{split}
\end{equation*}
for $\theta_k, \eta_k\in[0,1]$, $k\in\{1,2,...,d\}.$

Hence we can define the approximation solution $Y(t)$ similarly to
equation (\ref{eq:cas}), then we can derive the following theorem
analogically.
\begin{theorem}
Assume the SDEwMSs (\ref{eq:sdewms}) defined on $(\Omega,\mathcal{F},\{\mathcal{F}_t\}_{t\geq0},P)$ satisfying\\[.5cm]
\textbf{a}) $W(t),r(t)$ are independent,\\[.5cm]
\textbf{b}) $f(\cdot,\cdot),\ g(\cdot,\cdot),\bar{f}_\eta(\cdot,\cdot)$ satisfy conditions
$(\mathcal{H}1)$, $(\mathcal{H}2)$, ($\mathcal{H}$3$'$) and ($\mathcal{H}$4$'$),\\[.5cm]
then the unique strong solution $y(t)$ and numerical solution $Y(t)$
satisfying:
\[E(\sup_{0\leq t\leq T}|Y(t)-y(t)|^2)\leq C\Delta+o(\Delta),\]
where $C$ is a positive constant independent of $\Delta$.
\end{theorem}

\section{Numerical Stability }
In this section we consider numerical stability issues, extending
the analysis in Platen and Shi $\cite{PS8}$ to the Markovian
switching case. When simulating discrete time approximations of
solutions of SDEwMSs, numerical stability is clearly as important
as numerical efficiency. There have been various efforts made in the
literature trying to study numerical stability for a given scheme
approximating solutions of SDEs, see, for instance, Hofmann and
Platen $\cite{HP94}$, Higham $\cite{H0}$, Bruti-Liberati and Platen
$\cite{BP8}$ and Platen and Shi $\cite{PS8}$. Generally, for
analyzing numerical stability, some specifically designed test
equations are necessary to be introduced, the test SDEs used in the
above literatures are linear SDEs with multiplicative noise defined
as
\begin{equation}\label{eq:lsdes}
\begin{split}
dX_t=(1-\frac{3}{2}\alpha)\lambda X_tdt+\sqrt{\alpha |\lambda|}X_t dW_t,\\
\end{split}
\end{equation}
for every $t\geq 0$, where $X_0\geq0$, $\lambda<0$ and
$\alpha\in[0,1)$. Its explicit solution is of the form
\begin{equation}\label{eq:eslsdes}
\begin{split}
X_t=X_0exp\{(1-\alpha)\lambda t+\sqrt{\alpha |\lambda|}W_t\},t\geq 0\\
\end{split}
\end{equation}

As a unified criterion, Platen and Shi $\cite{PS8}$ proposed the
concept of $p$-stability criterion, which means that a process is
$p$-stable if in the long run its pth moment vanishes. Hence, for
$p>0, \lambda<0$, $p$-stable in equation (\ref{eq:lsdes}) may be characterized by
\[\lim_{t\rightarrow \infty}E(|X_t|^p)=0\quad \mathrm{iff}\quad 0\leq\alpha<\frac{1}{1+\frac{p}{2}}.\]
Since for different combinations of values of $\lambda\Delta,\alpha$
and $p$ with given time step size $\Delta$, a discrete time
approximation $Y_t$ and the original continuous process $X_t$ have
different stability properties. To explore these differences, the
concept of stability region is introduced. The stability region,
denoted by $\Gamma$, is determined by those triplets
$(\lambda\Delta,\alpha,p)$
$\in(-\infty,0)\times[0,1)\times(0,\infty)$
 for which the discrete time approximation $Y_t$  is
$p$-stable with time step size $\Delta$, when applied to the test
equation (\ref{eq:lsdes}).

By defining the random variable
\[G_{n+1}(\lambda\Delta,\alpha)=|\frac{Y_{n+1}}{Y_n}|,\]
which is called the transfer function of the approximation $Y_t$ at
time $t_n$, Platen and Shi $\cite{PS8}$ derive the following useful
result which can determine the stability regions for given schemes
by the following theorem.
\begin{theorem}
A discrete time approximation is for given $\lambda<0$,
$\alpha\in[0,1)$ and $p>0$, p-stable if and only if
\[E((G_{n+1}(\lambda\Delta,\alpha))^p)<1.\]
\end{theorem}

In the spirit of Platen and Shi $\cite{PS8}$, stability regions for
a range of schemes of SDEwMSs are discussed in this paper. Here we
consider the test process $X_{r(t)}=\{X_{t,r(t)},t\geq 0\}$
satisfies the linear SDEwMSs with multiplicative noise
\begin{equation}\label{eq:lsdewmss}
\begin{split}
dX_t=(1-\frac{3}{2}\alpha(r(t)))\lambda(r(t)) X_tdt+\sqrt{\alpha(r(t)) |\lambda(r(t))|}X_t dW_t,\\
\end{split}
\end{equation}
for every $t\geq 0$, where $r(t)$ is a Markov chain taking values in
a finite state space $S=\{1,2,...,N\}$,
 $r(0)=i_0\in S$, $X_0=X_{r(0)}\geq0$ and $(\alpha(r(t)),\lambda(r(t)))\in\{(\alpha(i),\lambda(i)),\alpha(i)\in[0,1), \lambda(i)<0,i=1,2,...,N\}$.
It is well known that the explicit solution of the test equation
(\ref{eq:lsdewmss}) is (see Mao and Yuan $\cite{MY6}$)

\begin{equation}\label{eq:eslsdewmss}
\begin{split}
X_t=X_0exp\{\int_{0}^{t}(1-\alpha(r(t)))\lambda(r(t)) dt+\int_{0}^{t}\sqrt{\alpha(r(t)) |\lambda(r(t))|}dW_t\}.\\
\end{split}
\end{equation}

For the convenience of numerical comparison, we introduce the
following stability criterion:
\begin{defi}\label{def:sps}
For $p>0$, a process $Y_{r(t)}=\{Y_{t,r(t)},t>0\}$ is called
state-p-stable if for each $i=1,2,...,N$, $Y_i=\{Y_{t,i},t>0\}$
satisfies
\[\lim_{t\rightarrow \infty}E(|Y_{t,i}|^p)=0.\]
\end{defi}

\begin{defi}\label{def:ssr}
The state-stability region $\Gamma_s$ is determined by those
triplets $(\lambda\Delta,\alpha,p)$
$\in(-\infty,0)\times[0,1)\times(0,\infty)$
 for which the discrete time approximation $Y_{r(t)}$ is
state-p-stable, when applied to the test equation
(\ref{eq:lsdewmss}), for each $i=1,2,...,N$,
$(\alpha(i),\lambda(i)\Delta,p)\in\Gamma_s$ with time step size
$\Delta$.
\end{defi}

Then we can obtain the following conclusion:

\begin{theorem}
The state-stability region $\Gamma_s$, which is generated by one
algorithm applied to the test equation (\ref{eq:lsdewmss}),  is the same as the
stability region $\Gamma$, which is generated by the same algorithm
applied to the test equation (\ref{eq:lsdes}).
\end{theorem}

\begin{proof}
$\Gamma\subseteq\Gamma_s$ is obvious. To prove
$\Gamma_s\subseteq\Gamma$, suppose for all $i=1,2,...,N$,
$(\alpha(i),\lambda(i)\Delta,p)\in\Gamma_s$. By Definition
\ref{def:ssr} we can see that  $Y_{r(t)}$ is state-$p$-stable. And
by Definition \ref{def:sps} we know for each $i=1,2,...,N$,
$Y_i=\{Y_{t,i},t>0\}=\{Y_{t,(\alpha(i),\lambda(i)\Delta,p)},t>0\}$
is $p$-stable, so $(\alpha(i),\lambda(i)\Delta,p)\in\Gamma$, for all
$i=1,2,...,N$. Immediately, we have $\Gamma_s\subseteq\Gamma$.
\end{proof}

\begin{Rem}By the conclusions derived in Platen and Shi
$\cite{PS8}$, we can also see that the PCEM methods are more
efficient than the EM method under these conditions in SDEwMSs case.
\end{Rem}

\section{Numerical examples}
In this section, we discuss two numerical examples to illustrate our
theory established in the previous sections. Let us now consider
several combinations of parameters $\theta$ and $\eta$ in equation
(\ref{eq:pcemc}), the names of the methods listed below are similar to those
used in Bruti-Liberati and Platen $\cite{BP8}$ and Platen and Shi
$\cite{PS8}$.

$(1) \theta=0, \eta=0$ (called EM scheme),

$(2) \theta=\frac{1}{2}, \eta=\frac{1}{2}$ (called symmetric PCEM
scheme),

$(3) \theta=\frac{1}{2}, \eta=0$ (called semi-drift-implicit PCEM
scheme),

$(4) \theta=1, \eta=0$ (called drift-implicit PCEM method),

$(5) \theta=0, \eta=\frac{1}{2}$ (called semi-diffusion-implicit
PCEM scheme),

$(6) \theta=1, \eta=1$ (called fully implicit PCEM scheme).

For a given problem, we will compare the simulated paths for these
different degrees of implicitness. If these paths differ
significantly from each other, then some numerical stability problem
is likely to be present and one needs to make an effort in providing
extra numerical stability for further research.

\begin{example}
Let W(t) be a scalar Brownian motion. Let r(t) be a right-continuous
Markov chain taking values in $S=\{1,2\}$ with generator
\[Q=\begin{bmatrix}-1 & 1 \\ q & -q \end{bmatrix}\]
And W(t) and r(t) are assumed to be independent. Consider an
one-dimensional linear SDWwMS
\begin{equation}\label{eq:ex1sdewms}
dy(t)=y(t)a(r(t))ds+y(t)b(r(t))dW(t)
\end{equation}
for $t\geq 0$, where
\begin{eqnarray*}\left\{
\begin{array}{l}
\displaystyle a(1)=1,\quad\quad b(1)=2\\[.5cm]
\displaystyle a(2)=2,\quad\quad b(2)=1\\
\end{array}
\right.
\end{eqnarray*}
\end{example}

It is well known that equation (\ref{eq:ex1sdewms}) has an explicit solution
\begin{equation}\label{eq:ex1essdewms}
y(t)=y_0\textrm{exp}[\int_{0}^{t}a(r(s))ds+\int_{0}^{t}(r(s))dW(s)-\frac{1}{2}\int_{0}^{t}b^2(r(s))ds].\\
\end{equation}

 For simulation reason, it is convenient to transform (\ref{eq:ex1essdewms})
 into following recursion form with $y(t_0)=y_0$,
\begin{equation}\label{eq:ex1enssdewms}
\begin{split}
y(t_{k+1})&
\approx y(t_k)\textrm{exp}[(t_{k+1}-t_k)a(r_{t_k})+(W({t_{k+1}})-W({t_k}))b(r_{t_k})\\
&\quad -\frac{1}{2}(t_{k+1}-t_k)b^2(r_{t_k})].\\
\end{split}
\end{equation}
Notice that $y(t_{k+1})$ in (\ref{eq:ex1enssdewms}) is not the exact value of $y(t)$
at the division points $t_{k+1}$, because $r(s)$ is not necessarily
constant on $[t_k,t_{k+1}]$. However, since
\[P\{r(t_{k+1})=i|r(t_k)=i\}=1+q_{ii}(t_{k+1}-t_k)+o(t_{k+1}-t_k)\rightarrow1\]
as $\Delta\downarrow0$, for sufficiently small $\Delta$, it is
reasonable to use (\ref{eq:ex1enssdewms}) as an approximation of the exact solution of
$y(t).\\$

$\mathbf{Case\ 1.}$ Let $q=2$, $y_0=200$, $r_0=1$, $\Delta=10^{-5}$,
$k=0,1,...,5\times 10^6$, namely for the corresponding time $0\leq
t\leq 50$. Compute the one-step transition probability matrix
\[Q(\Delta)=\begin{bmatrix}0.99999 & 0.00001 \\ 0.00002 & 0.99998 \end{bmatrix}\]
for the discrete Markov chain $r_{t_k}=r(k\Delta)$.\

By applying the previously described procedure, the trajectory of
the approximate solution $Y(t)$ with given stepsize $\Delta$ can be
constructed. In this paper, we do not draw the figure of the
simulating trajectory, instead, to carry out the numerical
simulation clearly, we repeatedly simulate and compute
$\displaystyle\sup_{t_k\in[0,50]}(|Y_{(\theta_i,\eta_i)}(t_k)-y(t_k)|^2)$
 ($i=1,2,3,4,5,6$) for 1000 times, then calculate the sample
mean $\widehat{E}(\displaystyle\sup_{t_k\in
[0,50]}|Y_{(\theta_i,\eta_i)}(t_k)-y(t_k)|^2)$($i=1,2,3,4,5,6$).
The results are listed in the following Table 1.\\

\begin{tabular}{|c|c|}\hline
$\quad\theta_i, \eta_i\quad$ & $\quad\quad\quad\quad\quad\widehat{E}(\displaystyle\sup_{t_k\in[0,50]}|Y_{(\theta_i,\eta_i)}(t_k)-y(t_k)|^2)\quad\quad\quad\quad\quad$ \\
\hline $0,0$                      & 1.90674403017316e+30\\
\hline $\frac{1}{2}, \frac{1}{2}$ & 7.84001334096008e+25\\
\hline $\frac{1}{2}, 0$           & 1.90607879136196e+30\\
\hline $1, 0$                     & 1.90541425316912e+30\\
\hline $0, \frac{1}{2}$           & 7.21187503884961e+25\\
\hline $1, 1$                     & 1.99576630102430e+30\\
\hline
\end{tabular}
\begin{center}\scriptsize{\textbf{Table 1}. Estimation of the errors between
the numerical solutions and exact solution\quad\quad}\\
\end{center}

$\mathbf{Case\ 2.}$ Let $q=1.5$, $y_0=200$, $r_0=1$,
$\Delta=10^{-5}$, $k=0,1,...,5\times 10^6$, namely for the
corresponding time $0\leq t\leq 50$. Compute the one-step transition
probability matrix
\[Q(\Delta)=\begin{bmatrix}0.99999 & 0.00001 \\ 0.000015 & 0.999985 \end{bmatrix}\]
for the discrete Markov chain $r_{t_k}=r(k\Delta)$.\

To carry out the numerical simulation we repeatedly simulate and
compute
$\displaystyle\sup_{t_k\in[0,50]}(|Y_{(\theta_i,\eta_i)}(t_k)-y(t_k)|^2)$
 ($i=1,2,3,4,5,6$) for 1000 times, then calculate the sample
mean $\widehat{E}(\displaystyle\sup_{t_k\in
[0,50]}|Y_{(\theta_i,\eta_i)}(t_k)-y(t_k)|^2)$($i=1,2,3,4,5,6$).
The results are listed in the following Table 2.\\

\begin{tabular}{|c|c|}\hline
$\quad\theta_i, \eta_i\quad$ & $\quad\quad\quad\quad\quad\widehat{E}(\displaystyle\sup_{t_k\in[0,50]}|Y_{(\theta_i,\eta_i)}(t_k)-y(t_k)|^2)\quad\quad\quad\quad\quad$ \\
\hline $0,0$                      &3.33550923616514e+36\\
\hline $\frac{1}{2}, \frac{1}{2}$ &1.71565111220697e+33\\
\hline $\frac{1}{2}, 0$           &3.34748775539284e+36\\
\hline $1, 0$                     &3.35948703848527e+36\\
\hline $0, \frac{1}{2}$           &3.51753180196352e+31\\
\hline $1, 1$                     &3.23681904239271e+36\\
\hline
\end{tabular}
\begin{center}\scriptsize{\textbf{Table 2}. Estimation of the errors between
the numerical solutions and exact solution\quad\quad}\\
\end{center}

$\mathbf{Case\ 3.}$ Let $q=0.5$, $y_0=200$, $r_0=1$,
$\Delta=10^{-5}$, $k=0,1,...,5\times 10^6$, namely for the
corresponding time $0\leq t\leq 50$. Compute the one-step transition
probability matrix
\[Q(\Delta)=\begin{bmatrix}0.99999 & 0.00001 \\ 0.000005 & 0.999995 \end{bmatrix}\]
for the discrete Markov chain $r_{t_k}=r(k\Delta)$.\

To carry out the numerical simulation we repeatedly simulate and
compute
$\displaystyle\sup_{t_k\in[0,50]}(|Y_{(\theta_i,\eta_i)}(t_k)-y(t_k)|^2)$
 ($i=1,2,3,4,5,6$) for 1000 times, then calculate the sample
mean $\widehat{E}(\displaystyle\sup_{t_k\in
[0,50]}|Y_{(\theta_i,\eta_i)}(t_k)-y(t_k)|^2)$($i=1,2,3,4,5,6$).
The results are listed in the following Table 3.\\

\begin{tabular}{|c|c|}\hline
$\quad\theta_i, \eta_i\quad$ & $\quad\quad\quad\quad\quad\widehat{E}(\displaystyle\sup_{t_k\in[0,50]}|Y_{(\theta_i,\eta_i)}(t_k)-y(t_k)|^2)\quad\quad\quad\quad\quad$ \\
\hline $0,0$                      &1.44025299136110e+71\\
\hline $\frac{1}{2}, \frac{1}{2}$ &4.66862461497885e+66\\
\hline $\frac{1}{2}, 0$           &1.39047511677566e+71\\
\hline $1, 0$                     &1.34154643977816e+71\\
\hline $0, \frac{1}{2}$           &5.41110836832358e+67\\
\hline $1, 1$                     &1.42636805044311e+71\\
\hline
\end{tabular}
\begin{center}\scriptsize{\textbf{Table 3}. Estimation of the errors between
the numerical solutions and exact solution\quad\quad}\\
\end{center}

This example has been studied in Yuan and Mao $\cite{YM4}$ in which
Case 1, Case 2 and Case 3 represent three different exponential
stability or instability situations respectively. However, since we
can easily show that the equation (\ref{eq:ex1sdewms}) does not
satisfy the conditions of equation (\ref{eq:lsdewmss}), it will be
not state-$p$-stable, the computer simulation results in Case 1,
Case 2 and Case 3 illustrate this point in some extent.
Nevertheless, when the parameters $\theta$ and $\eta$ are selected
reasonably, some PCEM methods are much more efficient
than the EM method in a certain extent.\\

\begin{example}
 Let r(t) be a right-continuous
Markov chain taking values in $S=\{1,2\}$ with generator
\[Q=\begin{bmatrix}-0.5 & 0.5 \\ 0.5 & -0.5 \end{bmatrix}\]
therefore, the one step transition probability from $r(t_k)$ to
$r(t_{k+1})$ is $e^{Q\Delta}$.
\end{example}

Consider the same 1-dimensional linear SDWwMS
\begin{equation}\label{eq:ex2sdewms}
dy(t)=y(t)a(r(t))ds+y(t)b(r(t))dW(t)
\end{equation}
on $t\geq 0$, this time let $a(r(t))$, $b(r(t))$take values as
follows
\begin{eqnarray*}\left\{
\begin{array}{l}
\displaystyle a(1)=0.15,\quad\quad  b(1)=0.1\\[.5cm]
\displaystyle a(2)=0.05,\quad\quad  b(2)=0.1\\
\end{array}
\right.
\end{eqnarray*}

Choose initial values $y_0=10$, $r_0=1$, $T$ is fixed at 10. By
applying the previously described procedure, the trajectory of the
approximate solution $Y(t)$ with given stepsize $\Delta$ can be
constructed.

To carry out the numerical simulation we successively choose the
stepsize $\Delta$ as the following Table 4, and for each $\Delta$,
we repeatedly simulate and compute
$\displaystyle\sup_{t_k\in[0,10]}(|Y_{(\theta_i,\eta_i)}(t_k)-y(t_k)|^2),$
 ($i=1,2,3,4,5,6$) for 1000
times, then calculate the sample mean
$\widehat{E}(\displaystyle\sup_{t_k\in
[0,10]}|Y_{(\theta_i,\eta_i)}(t_k)-y(t_k)|^2)$($i=1,2,3,4,5,6$).
The results are listed in the following Table 4.\\
\\
\begin{tabular}{|c|c|c|c|c|c|c|}\hline
\multicolumn{7}{|c|}{\bfseries $\quad\quad\quad\quad\quad\quad\quad\widehat{E}(\displaystyle\sup_{t_k\in
[0,10]}|Y_{(\theta_i,\eta_i)}(t_k)-y(t_k)|^2)\quad\quad\quad\quad\quad\quad\quad$}\\ \hline
$\Delta$ $\backslash$ $\theta_i,\eta_i$ &$\quad0,0\quad$&$\quad\frac{1}{2},\frac{1}{2}\quad$&$\quad\frac{1}{2},0\quad$&$\quad1,0\quad$&$\quad0,\frac{1}{2}\quad$&$\quad1,1\quad$\\
\hline $0.1$ &9702.4683&15.5672&1421.4360&8805.6057&7629.3505&7912.6245\\
\hline $0.01$ &376.9640&0.2299&265.3262&457.8189&150.4551&342.0570\\
\hline $0.001$ &24.0510&0.0017&23.6959&25.8235&1.1710&23.8747\\
\hline $0.0001$ &3.0465&2.16e-05&3.1024&3.1944&0.0158&3.0511\\
\hline $0.00001$ &0.1968&1.48e-07 &0.1969&0.1971&9.73e-05&0.1968\\
\hline
\end{tabular}
\begin{center}\scriptsize{\textbf{Table 4}. Estimation of the errors between
the numerical and exact solutions\quad\quad}\\
\end{center}

Clearly, we can easily show that the equation (\ref{eq:ex2sdewms})
satisfies the conditions of equation (\ref{eq:lsdewmss}), so it will
be state-$p$-stable, and the simulation results listed in Table 4
just illustrate the stability properties in a certain extent. On the
one hand, the numerical method reveals that the numerical solution
$Y(t)$ defined by the strong PCEM methods converge to the exact
solution $y(t)$ in $L^2$ as step size $\Delta\downarrow0$, and the
order of convergence is one-half, i.e.
\[E(\sup_{0\leq t\leq T}|Y(t)-y(t)|^2)\leq C\Delta+o(\Delta).\]
On the other hand, when the parameters selection are reasonable, the
PCEM methods is much more efficient than the EM method, which
strongly demonstrate our results.\\

\end{document}